\documentclass{amsart}
\usepackage{graphicx} 

\usepackage{amsmath}
\usepackage{amssymb}
\usepackage{mathrsfs}
\usepackage[english]{babel}
\usepackage{ot1patch}
\usepackage{xcolor}

\newtheorem{theorem}{Theorem}[section]
\newtheorem{lemma}[theorem]{Lemma}
\newtheorem{corollary}[theorem]{Corollary}
\newtheorem{proposition}[theorem]{Proposition}

\theoremstyle{remark}
\newtheorem{remark}[theorem]{Remark}

\newtheorem*{remark*}{Remark}

\theoremstyle{definition}
\newtheorem{definition}[theorem]{Definition}
\newtheorem{example}[theorem]{Example}

\numberwithin{equation}{section}

 \newcommand{\codim}{\operatorname{codim}}

\newcommand{\Rz}{\mathbb{R}}

\newcommand{\norm}[1]{\left\lVert#1\right\rVert}

\title[Directional curvature and medial axis]{Directional curvature and medial axis}
\author{Adam Bia\l o\.zyt}\address{AGH University of Krakow, Faculty of Applied Mathematics, al. Mickiewicza 30, 30-059 Krak\'ow, Poland}\email{bialozyt@agh.edu.pl}
\author{Dominik Bysiewicz}\address{Jagiellonian University, Faculty of Mathematics and Computer Science, \L ojasiewicza~6, 30-348 Krak\'ow, Poland}\email{dominik.bysiewicz@student.uj.edu.pl}
\author{Maciej P. Denkowski}\address{Jagiellonian University, Faculty of Mathematics and Computer Science, \L ojasiewicza~6, 30-348 Krak\'ow, Poland}\email{maciej.denkowski@uj.edu.pl}
\date{\today}

\begin{document}

\begin{abstract}
The medial axis $M_X$ of a closed set $X\subset \mathbb{R}^n$ is the set of points from the ambient space that admit more than one closest point in $X$. We study the problem of reaching the singularities, i.e. of characterising the points of the set $\overline{M_X}\cap X$. In order to tame the geometry, we assume that $X$ is definable in a polynomially bounded structure and obtain a general criterion based on a generalisation of the notion of {\it superquadraticity} previously introduced by Birbrair and Denkowski for $\mathscr{C}^1$-smooth hypersurfaces and extended to any codimension by Bia\l o\.zyt. We do not require any smoothness as we achieve our goal by introducing a notion of {\it directional curvature} in some naturally chosen {\it camber directions}. This allows us in particular to complete the study of the plane case. 
\end{abstract}
\maketitle

\centerline{\it Mihai Tib\u ar in memoriam}

\section{Introduction}

The notion of the {\it medial axis} $M_X$ of a given closed, proper subset $X\subset \mathbb{R}^n$ is well-known in pattern recognition as one of its main tools since its introduction by H. Blum in the sixties. It is defined as 
$$
M_X:= \{a\in \mathbb{R}^n\mid \# m(a) >1\},
$$
where $m(a):=\{x\in X\mid \|a-x\|=d(a,X)\}$ with the distance $d(a,X)$ from $a$ to $X$ calculated in the standard Euclidean norm. 

Even though medial axes have been studied extensively for quite a long time, their deep connection to singularity theory is a fairly recent discovery starting with \cite{D}, explored further in \cite{BD} (see also \cite{BDe} and \cite{DL} for a general overview), \cite{B}, \cite{BDD} (related also to \cite{Blne}, \cite{K} and especially \cite{Barx} which is the closest to the topic of our present paper. 

The main concern of the medial axis theory on the grounds of singularity theory is the question of {\it reaching the singularities}. In order to state it properly, let us introduce for $k\in \mathbb{N}$,
$$
\mathrm{Reg}_k X:=\{x\in X\mid \textrm{the germ}\ (X,x)\ \textrm{is}\ \mathscr{C}^k\textrm{-smooth}\}
$$
for the {\it $\mathscr{C}^k$-regular points} and 
$$
\mathrm{Sng}_k X:=X\setminus \mathrm{Reg}_k X
$$ 
for the set of {\it $\mathscr{C}^k$-singular points}. By what is known today as the {\it Nash Lemma} (see e.g. \cite{D}), 
$$
\overline{M_X}\cap \mathrm{Reg}_2 X=\varnothing,
$$
whence when looking for the intersection points $x\in \overline{M_X}\cap X$ we are compelled to study the $\mathscr{C}^2$-singular points of $X$. For any point $x\in X\cap \overline{M_X}$ of the intersection, the medial axis is said to {\it reach the singularity} $x$. Clearly, 
$$
\overline{M_X}\cap X\subset \mathrm{Sng_1 X}\cup (\mathrm{Reg}_1 X\cap \mathrm{Sng}_2 X)
$$
the union being disjoint. The main motivational picture behind the problem of reaching the singularities is given in Figure \ref{pic-par}.
\begin{figure}[h]
\includegraphics[width=8cm]{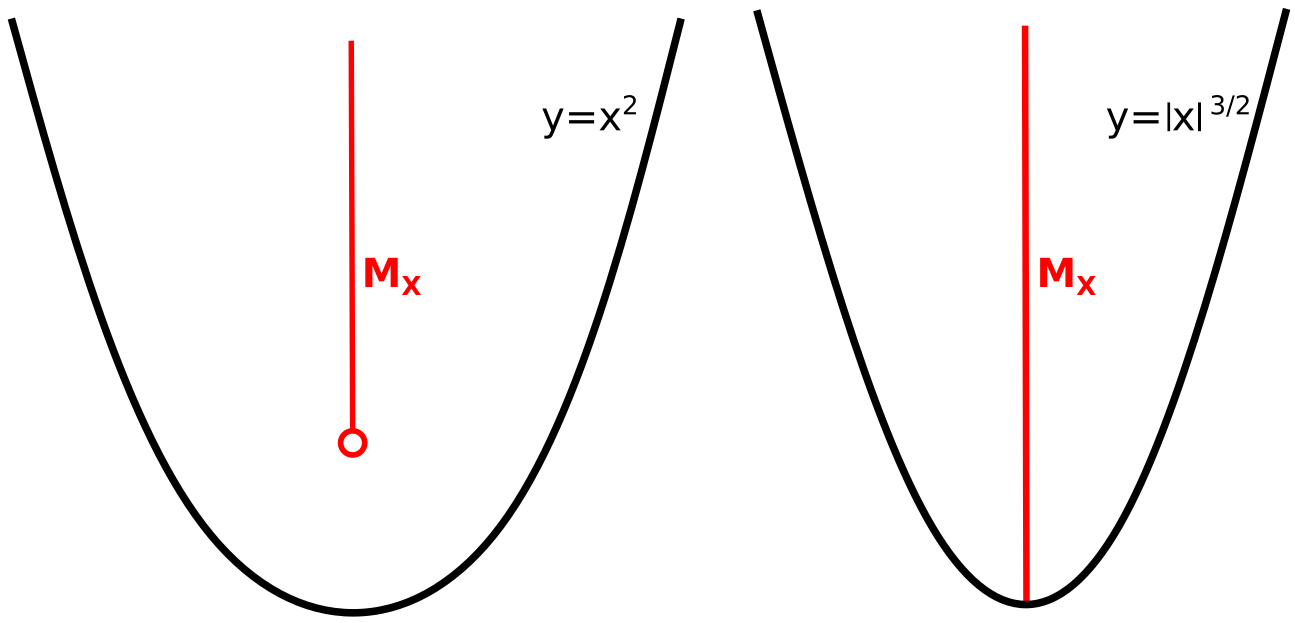}
\centering
\caption{Unlike the parabola, the curve $y=|x|^{3/2}$ is $\mathscr{C}^1$-smooth but not $\mathscr{C}^2$-smooth at the origin.}\label{pic-par}
\end{figure}

\smallskip
\textbf{Notation.}\\
We shall simplify the notation by putting
$$
X^{(1)}:=\mathrm{Sng_1 X}, \quad X^{(2)}:= (\mathrm{Reg}_1 X\cap \mathrm{Sng}_2 X).
$$

\smallskip
\textbf{Reaching Problem}\\
{\it    Characterize the points of the set $\overline{M_X}\cap X=\left(\overline{M_X}\cap X^{(1)}\right)\sqcup\left(\overline{M_X}\cap X^{(2)}\right)$.
}

\smallskip
As already noted, the Reaching Problem is thus split in two separate problems due to the fact that $X^{(1)}\cap X^{(2)}=\varnothing$. Studying the case of $\mathscr{C}^1$-singular and regular points usually calls for different methods. 

\smallskip
As a most appropriate illustration, consider the following simple examples (cf. \cite{DL})
 \begin{itemize}
    \item $X\colon y=|x|$ yields $X^{(1)}\cap \overline{M_X}=\{(0,0)\}=X\cap \overline{M_X}$;
    \item $X\colon y\geqslant |x|$ yields $X^{(1)}\cap \overline{M_X}=\varnothing=X\cap\overline{M_X}$;
    \item $X\colon y=|x|^{3/2}$ yields $X^{(2)}\cap \overline{M_X}=\{(0,0)\}=X\cap \overline{M_X}$;
    \item $X\colon y=\mathrm{sgn}(x)\cdot x^2$ yields $X^{(2)}=\{(0,0)\}$ but $(0,0)\notin\overline{M_X}$.
\end{itemize}

A rather surprising example is that of $X\colon y^2=x^3$ and even more that of a single branch $X\colon y=x^{3/2}, x\geq 0$ (see \cite{DL}) presented in Figure \ref{sch}. 
\begin{figure}[h]
\includegraphics[width=4.5cm]{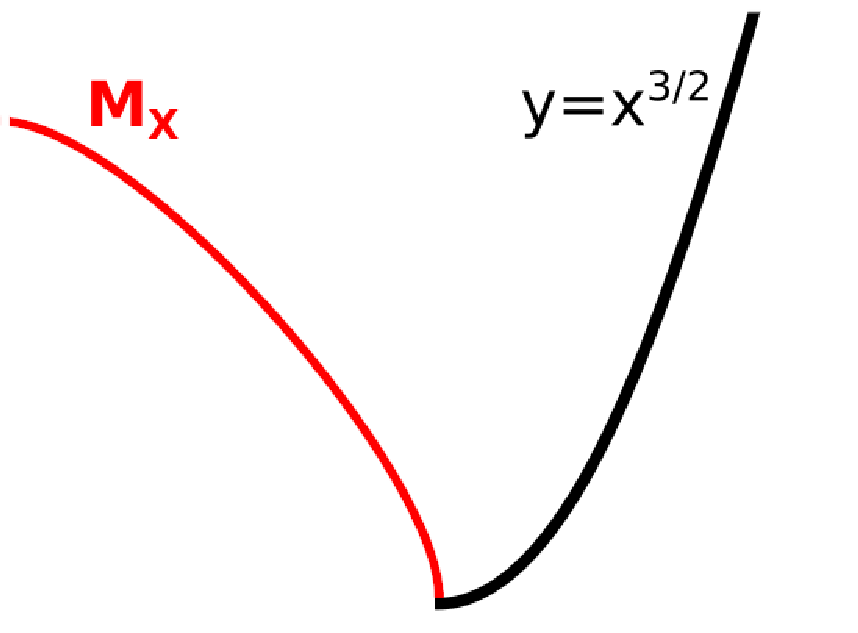}
\centering
\caption{The medial axis of the curve $y=x^{3/2}, x\geq 0$ has a `non-obvious' behaviour as a result of the curve's \textit{superquadraticity}.}\label{sch}
\end{figure}

The approaching of $M_X$ brings along also an interesting additional metric information about the singularity it reaches.


The natural setting for the study of the type of singularities reached by the medial axis is that of {\it tame geometry} i.e. subanalytic geometry or o-minimal structures. This is the point of view we will adopt throughout the paper, calling {\it definable} any set that is definable in a {\it polynomially bounded o-minimal structure}. For all these notions, we refer the reader to \cite{C} and the survey \cite{DD2}. As we are mainly interested in the local behaviour, there is no substantial difference between the subanalytic and the general definable case.

Up to now, only the case of subsets of $\mathbb{R}^2$ has been thoroughly understood; see \cite{BDe} and \cite{DL}; we complete this study in Sections \ref{plane geq} and \ref{plane <}. 

In the general case, we have the following {\it Tangent Cone Criterion} from \cite{BD} (see also \cite{Barx}) applicable only at singular points: if the {\it Peano tangent cone} of $X$ at $x$,
$$
C_xX:=\{v\in\mathbb{R}^n\mid \exists X\ni x_\nu\to x, \exists\lambda_\nu>0\colon \lambda_\nu(x_\nu-x)\to v\},
$$ 
is non-convex, then $x\in \overline{M_X}$ and, interestingly, $
        C_0(\overline{M_X})\supset\overline{M_{C_0X}}
        $ (with no equality, usually). Of course, this does not cover the case $x\in X^{(2)}$, 
        nor does it cover all possible cases for $x\in X^{(1)}$ 
        as the tangent cone might be linear even in the singular case. For example, $z^3=x^2+y^2$ describes a {\it horn surface} in $\mathbb{R}^3$ that is singular at the origin, the tangent cone degenerates to a half-line and the singularity is reached by the medial axis. On the other hand, the surface example of Ghomi and Howard 
        \cite{GH}: $z=xy(x^4+y^4)$ also has a $\mathscr{C}^1$-singularity at the origin reached by the medial axis, while the tangent cone at this point is a plane.

In \cite{BD} a description of the singular points reached by the medial axis is achieved through the introduction of the so-called {\it reaching radius} $r\colon X\to [0,+\infty]$ whose zero set is exactly $\overline{M_X}\cap X$ even in the non-definable case (the proof from \cite{BD} works in general). Nevertheless, we are still looking for a more detailed description of the type of singularities reached. A rather successful attempt to solve this problem for $X^{(2)}$ 
was made in \cite{BD} (see also \cite{DL}) in the case of a hypersurface $X$ by introducing the notion of {\it superquadratic points}. This is motivated by the situation encountered in the case of the plane. The notion of superquadraticity was carried over to $\mathscr{C}^1$-smooth points of any dimension in \cite{Barx} thanks to the fact that a submanifold is locally a graph over its tangent space. In essence, $X$ is called {\it superquadratic at} $a\in \mathrm{Reg}_1 X$ if the function 
$$
g_X(\varepsilon):=\max\{\|x-\pi(x)\|\mid x\in X\colon \|\pi(x)-a\|=\varepsilon\},\> 0\leqslant \varepsilon\ll 1,
$$
where $\pi\colon \mathbb{R}^n\to (T_aX+a)$ is the orthogonal projection onto the tangent space $T_a X$ translated to $a$ (see \cite{Barx}), satisfies
$$
\frac{g_X(\varepsilon)}{\varepsilon^2}\to +\infty\> (\varepsilon\to 0^+).
$$ The point is that in a polynomially bounded o-minimal structure, such a function is definable and can always be written as $$g_X(\varepsilon)=c\varepsilon^\eta+o(\varepsilon^\eta)$$ with $c,\eta>0$, provided its germ at the origin is non-zero (which is the case if $(X,a)$ is only $\mathscr{C}^1$-smooth  and not $\mathscr{C}^2$-smooth). The point $a$ being $\mathscr{C}^1$-regular, we have $\eta >1$ and superquadraticity means precisely that $\eta<2$. Since the exponent $\eta$ is just the order of vanishing $\mathrm{ord}_0 g_X$ of $g_X$ at the origin, it can be easily shown (cf. \cite{BD} Definition 3.10 and the remark that follows) that, in fact, assuming for simplicity $a=0$ and writing $\mathbb{R}^n=T_0 X\oplus (T_0 X)^\bot$ with the corresponding coordinates $x=u\oplus v$, 
$$
\mathrm{ord}_0 g_X=\max\{\theta>0\mid (X,0)\subset \{u\oplus v\mid \|v\|\leqslant \mathrm{const.}\|u\|^\theta\}\}.
$$
Then, if we put $$\mathcal{SQ}(X):=\{x\in \mathrm{Reg}_1 X\mid X\ \textrm{is superquadratic at}\ x\},$$ we have, of course, $\mathcal{SQ}(X)\subset X^{(2)}$ 
and the following theorem holds (\cite{BD} for $\codim X=1$, \cite{Barx} for the general case).
\begin{theorem}
    In the situation considered above, for any $a\in X^{(2)}$, 
    $$
    a\in \mathcal{SQ}(X)\Rightarrow\ a\in \overline{M_X}.
    $$
\end{theorem}
As shown in \cite{Barx}, the converse is generally true only in the case of $\dim X=1$.

\medskip
\noindent\textbf{Aims and results}

The aim of the present paper is to extend the notion of superquadratic points and the theorem quoted above also to $\mathscr{C}^1$-singular points. We shall do this under some mild assumptions that encompass the previous example of a horn (Example \ref{horn}). We obtain a general notion that is valid for $X^{(2)}$ and $X^{(1)}$ as well, and gives new insight into the geometry of the germ through a new criterion for the reaching of the singularity. 

Our main results are Theorems \ref{kierunek1} coupled with \ref{kierunek} and the Corollary \ref{wniosek}. They feature the {\it directional curvature} and the {\it general curvature} we introduce in Definitions \ref{dc} and \ref{gc}, respectively. They are well posed under the mild assumption $(*)$ from Section \ref{section dc}. These results are preceded by the case of the plane in Sections \ref{plane geq} with Theorem \ref{plane geq 2} and Section \ref{plane <} with a straightforward proof of Theorem \ref{plane <2}. Moreover, the plane allows for a result, Theorem \ref{plane geq 2}, that cannot be extended to higher dimensions as the example constructed in Proposition \ref{higher geq 2} shows.

\section{Directional curvature}\label{section dc}

For all of the following results we assume the closed, definable germ $(X,0)\subsetneq (\mathbb{R}^n,0)$ to satisfy the following assumption: the affine hull of the Peano tangent cone of $X$ at zero is of the form $$\mathrm{Aff}(C_0X)= \mathbb{R}^k\times \{0\}^{n-k},\leqno{(*)}$$
where $1\leqslant k<n$. Note that there is no simple relation between $k$ and $\dim_0 X$, even though $\dim C_0X\leq \dim_0 X$.
\begin{remark}\label{pierwsza}
    Assuming $(*)$ yields in particular $C_0X\cap \{0\}^k\times\mathbb{R}^{n-k}=\{0\}^n$ which is equivalent to the inclusion 
    $$
    X\cap U\subset\{(x,y)\in \mathbb{R}^k\times\mathbb{R}^{n-k}\mid \|y\|\leqslant C\|x\|\},
    $$
    for some small neighborhood $U$ of $0\in\mathbb{R}^n$ and some constant $C>0$.
\end{remark}
    

In addition, but only for $n=2$, we will require that $X$ be `filled horizontally'; instead of considering $X_1:=\{(x,y)\mid y^2=x^4\}$, we prefer to turn the attention to $X_2:=\{(x,y)\mid y^2\leqslant x^4\}$. Filling in the part of the plane that lies beyond $X$ in horizontal direction prevents $M_X$ from approaching $0$ from the tangent directions ( cf. $0\in\overline{M_{X_1}}$, whereas $0\notin\overline{M_{X_2}}$), which is key in the later results.
\begin{figure}[h]
\includegraphics[width=8cm]{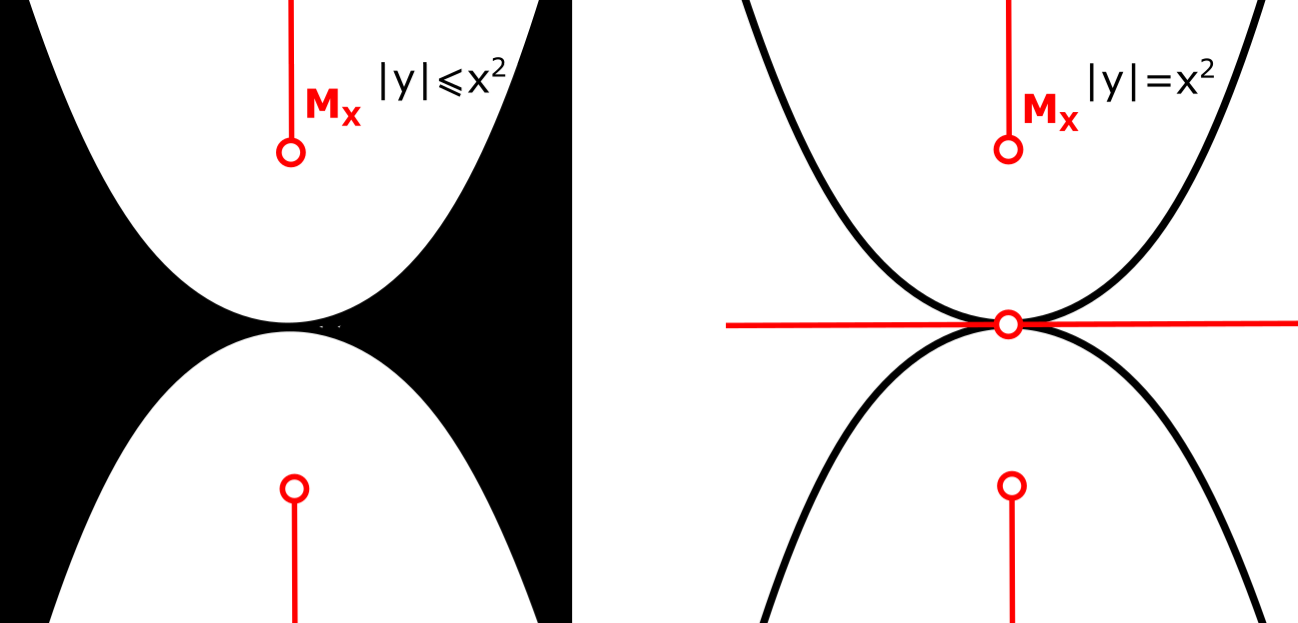}
\centering
\caption{If the $X$ is not horizontally filled
, then naturally the origin is achievable by $M_X$
, for any $X$ satisfying $(\ast)$.}
\end{figure}

To define things precisely, let $B(0,2r)$ be the disk of radius $2r>0$, centered at $0$. Thanks to $(\ast)$, we can take the radius $r$ small enough to get $B(0,2r)\cap X \cap\{x=0\}=\{0\}$. We will call the set $X$ {\it filled horizontally} if for any $r>0$ smaller than some initial $r_0>0$, the set $B(0,2r)\setminus X$ has at most two connected components. One containing $(0,r)$ and (possibly the other) containing $(0,-r)$.

Let us recall the notion of the {\it normal cone} to $X$ at $0\in X$:
$$
N_0(X):=\{v\in\mathbb{R}^n\mid \forall w\in C_0X,\ \langle v,w \rangle\leqslant 0\}.
$$
We will consider {\it normal directions} i.e. vectors from $$V_0(X):=N_0(X)\cap \mathbb{S}^{n-1},$$ or rather a subset of this set.

\begin{remark}\label{jeden} In the situation introduced above, for any $v\in \{0\}^k\times \mathbb{R}^{n-k}\cap \mathbb{S}^{n-1}$, we have $v\in V_0(X)$. Moreover, $\mathbb{R}v\cap C_0X = \{0\}$ which implies that considering the decomposition $\mathbb{R}^n = (\mathbb{R}v)^{\perp} \oplus \mathbb{R}v$, there exists $c>0$ such that:
$$
(X,0)\subset \{z\oplus tv\in (\mathbb{R}v)^{\perp} \oplus \mathbb{R}v\mid |t|\leqslant c\norm{z}\},
$$
which is to be understood in the sense that a representative of the germ is contained in the set on the right-hand side (as in Remark \ref{pierwsza}).
\end{remark}
This remark allows us to state the following definition in which we use
$$
P_v(\theta,c):=\{z\oplus tv\in (\mathbb{R}v)^{\perp} \oplus \mathbb{R}v\mid |t|\leqslant c\norm{z}^\theta\}
$$
defined for $\theta>0$ and $c>0$.

\begin{definition}\label{dc} Under the assumption $(*)$ we define the \textit{directional curvature} in the direction $v\in V_0^C (X):=\{0\}^k\times \mathbb{R}^{n-k}\cap \mathbb{S}^{n-1}$ of $(X,0)$ to be 
$$
\theta_v(X,0) := \sup\{\theta>0\mid  (X,0)\subset P_v(\theta,c), \ \textrm{for some}\ c>0\}\in [1,+\infty].
$$
We call the directions $v\in V_0^C (X)\subset V_0(X)$ used here {\it camber directions}.
\end{definition}

\begin{lemma}\label{skoncz}
    Given $v\in V_0^C (X)$, the least upper bound $\theta_v(X,0)$ is finite, unless $(X,0)\subset (\mathbb{R}v)^\bot$.
\end{lemma}
\begin{proof}
    By the Curve Selection Lemma we find a definable curve $\gamma\colon [0,\varepsilon)\to
X$ with $\gamma\cap (\mathbb{R}v)^\bot=\{0\}$. Now, $\gamma(t)=(z(t)\oplus\tau(t)v)\in(\mathbb{R}v)^\bot\oplus \mathbb{R}v$ means that after identifying the direct sum with the product $H\times \mathbb{R}$ where $H=(\mathbb{R}v)^\bot$ we may reparametrize $z(t)$ by the distance to the origin. Indeed, the definability of $z(t)$ implies that for all sufficiently small $t>0$, $z(t)\cap t\mathbb{S}^{n-2}$ reduces to a single point. Next, having $\|z(t)\|=t$ allows us to identify the curve $\gamma(t)$ with $\eta(t)=(t,\tilde{\tau}(t))$. Then it is easy to see that $\mathrm{ord}_0\eta(t)$ (understood as the minimal order of the components $\eta=(\eta_1,\dots,\eta_n)$) is an upper bound for $\theta_v(X,0)$ and by construction this $\mathrm{ord}_0\eta(t)$ is finite.
\end{proof}
\begin{remark}\label{osiaganie}
Unless $(X,0)\subset (\mathbb{R}v)^\bot$, the least upper bound is attained and belongs to the field of exponents $\mathbb{F}$ of the o-minimal structure (which is $\mathbb{Q}$ in the globally subanalytic or semi-algebraic cases). In fact, if we consider the two projections $\rho(z\oplus tv)=z$ and $p(z\oplus tv)=t$, we see that $\rho^{-1}(0)\cap X\subset p^{-1}(0)\cap X$ (cf. Remark \ref{jeden}). This readily implies that these two functions satisfy on $X$ the \L ojasiewicz inquality (cf. \cite{DD2}) and, in fact, the number $\theta_v(X,0)$ corresponds to their \L ojasiewicz exponent, which we know to be attained (see \cite{DD2} or \cite{BD} and the references to this classical Bochnak-Risler result therein).
\end{remark}

\begin{lemma}\label{tangentcomponent}
    The directional curvature $\theta_v(X,0)$ in the camber direction $v$ is equal to \[\sup\{\theta>0\mid x\in X\Rightarrow |\langle v,x\rangle|\leqslant \mathrm{const.}\|\pi(x)\|^\theta\},\] where $\pi$ denotes an orthogonal projection on $\mathrm{Aff}(C_0X)$.
\end{lemma}
\begin{proof}
Observe first that the lemma holds when $(X,0)\subset (\mathbb{R}v)^\perp$. In that case, the supremum and $\theta_v(X,0)$ are simply equal to infinity. For the rest of the proof let us assume that $(X,0)\not\subset(\mathbb{R}v)^\perp$. 

Take a decomposition of $\mathrm{Aff}(C_0X)^\perp = U\oplus \mathbb{R}v$ and denote $x\in X$ as $x=x_1\oplus x_2\oplus tv\in \mathrm{Aff}(C_0X)\oplus U \oplus \mathbb{R}v$. We have then $x_1=\pi(x)$ and $t=\langle v,x\rangle$.

It is plain to see that $\|x_1\|^\theta\leqslant \|z\|^\theta$ for any positive exponent $\theta$, thus \[\sup\{\theta>0\mid x\in X\Rightarrow |\langle v,x\rangle|\leqslant \mathrm{const.}\|\pi(x)\|^\theta\}\leqslant \theta_v(X,0).\]

To prove the opposite inequality, we first follow the Remark~\ref{osiaganie}, to obtain a positive constant $c$ such that $|t|\leqslant c\|z\|^{\theta_v(X,0)}$. Moreover, Remark~\ref{jeden} allows us to estimate $\|z\|\leqslant \|x_1\|+\|x_2\|\leqslant \tilde{c}\|x_1\|$. For any $x_1\oplus x_2\oplus tv\in X$, we then have
\[|t|\leqslant C\|x_1\|^{\theta_v(X,0)}\]
with $C=c\tilde{c}^{\theta_v(X,0)}$. Consequently
\[ \theta_v(X,0) \leqslant \sup\{\theta>0\mid x\in X\Rightarrow |\langle v,x\rangle|\leqslant \mathrm{const.}\|\pi(x)\|^\theta\}.\qedhere\]
\end{proof}

\begin{definition}\label{gc} The \textit{general curvature} of the germ $(X,0)\subset \mathbb{R}_x^k\times \mathbb{R}_y^{n-k}$ satisfying $(*)$ is defined to be 
$$
\theta(X,0):= \sup \left\{\theta>0\mid (X,0)\subset \left\{\norm{y}\leqslant \mathrm{const.}\norm{x}^{\theta}\right\}\right\}.
$$
\end{definition}
\begin{lemma}
Unless $(X,0)\subset \mathbb{R}^k\times \{0\}^{n-k}$, the supremum in the definition of $\theta(X,0)$ is finite and attained in $\mathbb{F}_{\geqslant 1}$.
\end{lemma}
\begin{proof}
Put $\pi_k(x,y)=x$ and $\pi_{n-k}(x,y)=y$. Then by Remark \ref{pierwsza}, $\pi_k^{-1}(0)\cap X\subset \pi_{n-k}^{-1}(0)\cap X$ and as earlier it is easy to see that $\theta(X,0)$ corresponds to their \L ojasiewicz exponent. Hence, the result.
    \end{proof}

\begin{definition} We call the germ $(X,0)$ \textit{superquadratic} if $\theta(X,0)<2$.\end{definition}

This encompasses the polynomially bounded cases considered in \cite{BD} and \cite{Barx}. Note also that the definition somehow fixes the coordinate system.

\begin{proposition}
    If $0\in\mathrm{Reg}_2 X$, then $\theta(X,0)\geqslant 2$.
\end{proposition}
\begin{proof}
    Since $C_0X$ coincides with the tangent space $T_0 X$, we conclude that $(*)$ reads $T_0X=\mathbb{R}^k\times\{0\}^{n-k}$ with $k=\dim_0 X$. Then $(X,0)$ is the graph of a $\mathscr{C}^2$-smooth map germ $f\colon (\mathbb{R}^k,0)\to (\mathbb{R}^{n-k},0)$ of the first $k$ variables with differential $d_0 f=0$. Therefore, $f(x)=\frac{1}{2}d_0^2f(x,x)+o(\|x\|^2)$ which implies $\|f(x)\|\leqslant\mathrm{const.}\|x\|^2$ in a neighbourhood of the origin and the assertion follows.
\end{proof}

As we shall need the {\it reaching radius} $r\colon X\to [0,+\infty]$, let us briefly recall its definition from \cite{BD} (see also \cite{DL}):  
\begin{definition}\label{rr} We first define the {\it weak reaching radius} of $X$ at $a\in X$ by 
$$
r'(a)=\inf_{v\in V_a(X)} r_v(a)
$$
where 
$$
r_v(a)=\sup\{t\geqslant 0\mid a\in m(a+tv)\}
$$
is the {\it directional reaching radius}. Next, we put
$$
\tilde{r}(a)=\liminf_{X\setminus\{a\}\ni x\to a} r'(x)
$$
for the {\it limiting reaching radius}. Finally, we define the {\it reaching radius} as 
$$
r(a)=\begin{cases}
r'(a), &a\in\mathrm{Reg}_2 X,\\
\min\{r'(a),\tilde{r}(a)\}, &a\in\mathrm{Sng}_2 X.
\end{cases}
$$
\end{definition} 
This definition has recently been exploited and simplified to some extent in \cite{B}. The main thing we are interested in is the following result that does not require in fact any definability assumption (the proof given in \cite{BD} works for any set $X$ as below).
\begin{theorem}[\cite{BD}, Theorem 4.35]\label{radius}
  For any closed, nonempty proper subset $X$ of $\mathbb{R}^n$,  $r^{-1}(0)=\overline{M_X}\cap X.$
\end{theorem}

Our first main result is the relation between the general curvature and the directional curvatures. 

\begin{theorem}\label{theta}
    $\theta(X,0)=\min\{\theta_v(X,0)\mid v\in \{0\}^k\times\mathbb{R}^{n-k}\cap \mathbb{S}^{n-1}\}.$
\end{theorem}
\begin{proof}
    Let $v_1, \ldots ,v_k\in\mathbb{S}$ span the $\mathrm{Aff}(C_0X)^\perp$. We immediately obtain \[\inf \{\theta_v(X,0) \mid v\in  \mathrm{Aff}(C_0X)^\perp \cap\mathbb{S}\}\leqslant \min_{i=1,\ldots, k}\{\theta_{v_i}(X,0)\}.\] 
    
    Conversely, we take any $v\in \mathrm{Aff}(C_0X)^\perp \cap \mathbb{S}$ and assume that $\min_{i=1,\ldots, k}\{\theta_{v_i}(X,0)\}$ $>\theta_v(X,0)$. Then there exists $\delta>0$ such that $\theta_{v_i}(X,0)>\delta>\theta_v(X,0)$ for any $i=1,\ldots,k$. 
    
    Let us write $v=\sum \lambda_iv_i$, take a decomposition of $\mathrm{Aff}(C_0X)^\perp = U\oplus \mathbb{R}v$ and denote $x\in X$ as $x=x_1\oplus x_2\oplus tv\in \mathrm{Aff}(C_0X)\oplus U \oplus \mathbb{R}v$. Since $\theta_{v_i}(X,0)>\delta$, we can find constants $c_i$ for which $|t |\leqslant c_i\|x_1\|^\delta$ by Lemma~\ref{tangentcomponent}. 
    
    We are now ready to bound the norm of $z=x_1\oplus x_2$ as follows
%
%
    
    \begin{align*}
    \|z\|^\delta=\mathrm{const.}\sum \lambda_i c_i\|z\|^\delta &\geqslant \mathrm{const.}\sum \lambda_i c_i\|x_1\|^\delta\geqslant\\
    &\geqslant \mathrm{const.}\sum \lambda_i \|tv_i\| \geqslant \mathrm{const.}\|\sum \lambda_i tv_i\|=\mathrm{const.}\|tv\|.
    \end{align*} 
 Thus, we obtain a contradiction with the choice of $\theta_v(X,0)$.
\end{proof}

\section{Plane case and $\theta\geqslant 2$}\label{plane geq}
Let us recall that we assumed the closed, definable planar germ $(X,0)\subsetneq (\mathbb{R}^2,0)$ to satisfy two conditions:
\begin{enumerate}
    \item[$(\ast)$] the affine hull of the Peano tangent cone of $X$ at zero is of the form $$\mathrm{Aff}(C_0X)= \mathbb{R}^k\times \{0\}^{n-k},$$
    \item[$(\ast\ast)$] the germ $X$ is horizontally filled.

\end{enumerate}

Let us put $\mathbb{R}_\pm := \{x\in \mathbb{R}\mid \pm x\geqslant 0\}$ and for any set $X\subset \mathbb{R}^2$ let $X^\pm:=X\cap \{(x,y)\in\mathbb{R}^2\mid \pm x\geqslant 0\}$.

We shall use a notion of the {\it central set} $C_X$ of $X$. It is the set of all the centres of {\it maximal balls} contained in $\mathbb{R}^n\setminus X$, i.e. such open balls $B\subset \mathbb{R}^n\setminus X$ that are not proper subsets of any other open ball $B'\subset \mathbb{R}^n\setminus X$. The notion of the central set is closely related to that of the medial axis. Indeed, a sequence of inclusions holds: $M_X\subset C_X\subset \overline{M_X}$ (see e.g. \cite{BD}). The elements of $C_X$ are conveniently detected by the directional reaching radius. Indeed, one can easily verify that for any $a\in X$ and $v\in V_a(X)$ with finite directional reaching radius, there is $a+r_v(a)(v-a)\in C_X.$

Recall also that a definable, connected set is pathwise connected.

\begin{lemma}\label{wnetrze} Let $X\subset \mathbb{R}^2$ be a connected, definable set with a nonempty medial axis $M_X$. Let $p\in M_X$ be such that $m(p)$ is not connected and let $a, b$ be points in two different components of $m(p)$. Suppose there is a path joining $a$ and $b$ in $X$ that lies entirely inside the sector defined by the acute angle between the segments  $[p,a], [p,b]$. Then we define $S\subset \mathbb{R}^2$ to be the corresponding closed area limited by segments $[p,a], [p,b]$ and by $\partial X$. Then $M_X\cap \mathrm{int} S\neq \varnothing$.\end{lemma}

\begin{proof} 
    Let us consider any closed disk inside $S$, centered at a point $o\notin M_X$ and intersecting $X$ at the unique closest point that we want to be $q\in S\cap X$. Such a point clearly exists. Then the half-line $\mathbb{R}_{+}(o-q)$ with the origin at $q$ must intersect $C_X$. Indeed, the directional reaching radius $r_{v}(q)$ in the normalized direction given by $o-q$ exceeds $\|o-q\|$. Then, due to definability, the half-line intersects $\partial S$ at a finite number of points. Denote by $y$ the point of the intersection that is closest to $q$. Then either $y\in \partial X\setminus\{q\}$, or $m(y)\in\{a,b\}$. In either case $q$ is not the closest point to $y$. Therefore, in the open segment $(q,y)$ there must exist an element of $C_X$, which belongs to $\mathrm{int} S$. But arbitrarily near this point we can find an element of $M_X$ and this ends the proof.
\end{proof}

\begin{figure}[h]
\includegraphics[width=6cm]{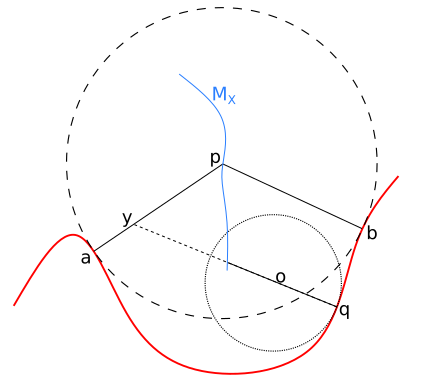}
\centering
\end{figure}

\begin{remark} If $(X,0)\subset (\mathbb{R}^2,0)$ is not superquadratic, then for some neighborhood $V\ni 0$, there is $V\cap M_X= \varnothing$ (e.g., in the $\mathscr{C}^2$-smooth case, one can take $V=B(0,\norm{f}/2)$, where $f$ is a focal point).\end{remark} 

We are assuming now $X$ connected and satisfying $(*)$ and  $(**)$. Notice that in such a case, the germ of $\partial X$ is formed by at most two planar curves.  

\begin{theorem} If $0\in \overline{(M_{X^+})}$, then there exists a neighbourhood $V\ni 0$ such that  $$\forall x\in M_{X^+}\cap V, \  0\in m(x).$$\end{theorem}
\begin{proof}
    Observe first that since $X$ is definable, so are $X^+$ and $M_{X^+}$.
    
    Let us suppose by contradiction that there exists a sequence $\{p_n\}$ of points in $(M_{X^+})^+$ such that $\norm{p_n}\to 0$ but with $0\notin m(p_n)$.
    \par For all of those points consider circles $C_n:=C(p_n, d(p_n, X))$. Since $0$ lies outside of them all (as it is not the closest point to any $p_n$), 
    by passing to a subsequence we can ensure that whenever $i\neq j$, we have $C_i\cap C_j=\varnothing$. Indeed, if we chose pairwise disjoint circles $C_1,\ldots, C_k$, then we can choose a neighbourhood $V:=B(0,2\varepsilon)$ of $0$ disjoint with them, as $0$ lies outside of all the circles. We can now choose $C_{k+1}$ as a circle with center at any of  $p_{k+1}\in B(0,\varepsilon)$. Notice that by a slight movement of the points $p_n$ we can also assume that $C_n\cap X^+ = m(p_n)$ are not connected spaces. Indeed, proceeding iteratively there can only be two situations:
    \vspace{10pt}\par (1) $C_n\cap X$ is not connected;
    \par Then the claim holds without further adjustments of $p_n$;
    %
    \vspace{10pt}\par (2) $C_n\cap X$ is connected;
    \par Assume that $n$ is the first index for which $C_n\cap X$ is connected. Then let $a,b$ be end points of the arc $C_n\cap X^+$ (since $X^+$ is horizontally filled $C_n\cap X^+$ cannot be the whole circle). Consider $C_n'$ defined as a circle passing through $a,b$ centered at point $p_n'$ with $\norm{p_n-p_n'}<\varepsilon_0$. Moreover, let $C_n''$ be a circle concentric with $C_n'$ and with radius $d(p_n', X^+)+\varepsilon$, where $\varepsilon>0$ is chosen so that for the disk $D''_n$ having $\partial D_n''=C_n''$, we get $D_n''\cap X^+$ disconnected (this is possible for $\varepsilon$ small enough).
    If we choose $\varepsilon_0$ and $\varepsilon$ small enough, then $C_n''$ is still disjoint with the other circles $C_k$. Since $D_n''\cap X$ has at least $2$ components, a point $q\in M_{X^+}$ exists inside $D_n''$ such that $C(q,d(q,X^+))\subset D_n''$, so we can replace $p_n$ with $q$ to obtain the previous case.
    

    \begin{figure}[h]
    \includegraphics[width=6cm]{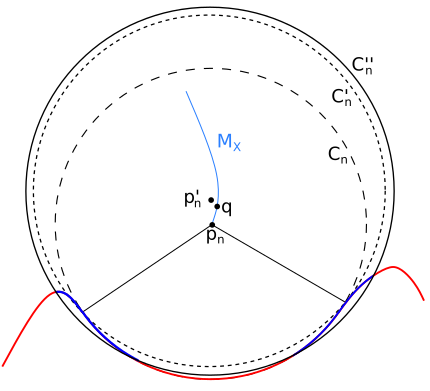}
    \centering
    \end{figure}
    
    \vspace{10pt}\par Having $C_n\cap X$ disconnected we can choose two points $a,b\in C_n\cap X$ that belong to different connected components of this set. Then we consider the open subset $S_n\subset \mathbb{R}^2$ limited by $\partial X$ and the segments $[p_n,a], [p_n,b]$ (it is well defined, because $\partial X$ is connected). According to Lemma \ref{wnetrze} there exists a neighbourhood $V\ni p_n$, such that $(S_n\cap V)\cap M_X\neq \varnothing$.
    

    Passing once more to a subsequence, we can also assume that for indices $i<j$, all points in $m(p_i)$ appear before those of $m(p_j)$ on the curve that forms $\partial X^+$. In such a case, the regions $S_n$ are pairwise disjoint.
    
    We are now getting close to the final contradiction. On the one hand, we know that $M_{X^+}$ is a boundary set formed by a finite family of curves. On the other hand, every region $S_n$ contains a distinct branch of $M_X$. However, since $\mathbb{R}^2\setminus X^+$ is simply connected, $M_{X^+}$ does not contain any loops (see \cite{Lieu}) and any curve of the initial family of curves that build $M_X$ can intersect with up to two different regions $S_n$ --- a contradiction with definability of $M_X$ since the number of curves that form $M_X$ would have to be infinite.
\end{proof}

\begin{figure}[h]
\includegraphics[width=12cm]{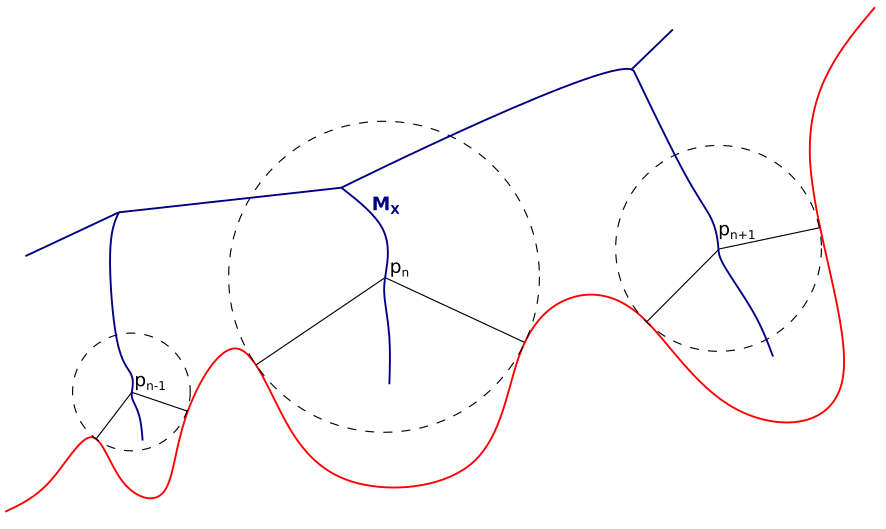}
\centering
\end{figure}

\begin{remark}If $X$ is not superquadratic and satisfies conditions $(\ast)$ and $(\ast \ast)$, then there exists a neighbourhood $V\ni 0$, such that for any $x\in M_X\cap V$ one has $m(x)\subset X^+$ or $m(x)\subset X^-$. Together with the previous result this remark leads to the following theorem. \end{remark}

\begin{theorem}\label{plane geq 2} For a definable $X$ satisfying $(\ast)$ and $(\ast \ast)$, $\theta(0,X)\geqslant 2$ implies $0\notin \overline{M_X}$.\end{theorem}
\begin{proof}
    Since $\theta\geqslant 2$, then, in particular, $(X,0)$ is bounded by $P:= \{(x,\pm ax^2),\ x\in\mathbb{R}\}$ for some $a>0$. Thanks to the properties of a parabola there exists a neighbourhood $V\ni 0$, such that for $x\in V$ and any $r>0$ an implication holds
    $$
    B(x,r)\cap P^- \neq\varnothing\ \land\ B(x,r)\cap P^+ \neq\varnothing \qquad \implies \qquad 0\in B(x,r).
    $$
    Hence, if $x\in V$ and $m(x)\not\subset X^+$ and $m(x)\not\subset X^-$, then $0\in m(x)$. However, since $(X,0)^+$ and $(X,0)^-$ are nonempty, $0\in m(x)$ is possible only for $x\in \{0\}\times \mathbb{R}$. Finally, for $x\in V\cap \{0\}\times \mathbb{R}$ we have $\{0\}=m(x)$ when $V$ is small enough. 

    In summary, if there exists a sequence $\{p_n\}\subset M_X$ such that $p_n\to 0$, then by passing to a subsequence we can assume that $m(p_n)\subset X^+$ (or $X^-$) for all points. However, it is impossible by Theorem 3.3.
\end{proof}

\section{Plane case and $\theta<2$}\label{plane <}

In the previous section, we proved that sequences of points $x\in M_X$, such that $m(x)\subset X^+$ or $m(x)\subset X^-$, cannot converge to zero (the proof is independent of assumption $\theta \geqslant 2$). Therefore, points of a sequence in $M_X$ convergent to zero must have the closest points in both parts of $X$ simultaneously. Now, we will prove that if $\theta<2$ then $0\in \overline{M_X}$.

\begin{theorem}\label{plane  <2} For a definable $X$, $\theta< 2$ implies $0\in \overline{M_X}$.\end{theorem}
\begin{proof}
    
    Let us set $c>0$ to be the infimum of constants for which the germ of $X$ is a subset of $\{(x,y)\in\Rz^2\mid y=\mathrm{const.}x^\theta\}$\footnote{Notice here that although the constant $c$ is greater than zero due to the polynomial bound of the o-minimal structure, it might still be unattained. Consider for example $X$ equal to the graph of $y=x^{3/2}+x^2.$}. Pick now any $0<\tilde{c}<c$ and observe that arbitrary close to $0$ we can find a point $p$ on the curve $y=\tilde{c}\,x^\theta$ lying under $\partial X$.

    \par Let us consider now closed disks $B_r:=B(o_r,r)$ tangent to the right part of the graph of $y=\tilde{c}\, x^\theta$ at the point $p$. According to the properties of the curve $y= \tilde{c}\,x^\theta$ for $\theta<2$, we know that the centre of $B_r$, the point $o_r$, lies in the left halfplane for all radii $r$ such that $0\in B_r$. Hence, denoting by $B:=B(o,r')$ the disk with minimal radius among those with $0\in B_r$, there exists a neighbourhood $V$ of zero that $V\cap X^-\subset B$.
    \par Now, let us consider disk $B':=B(o, r'-\delta)$ concenter with $B$ having a bit smaller radius, such that both $B'\cap X^-$ and $B'\cap X^+$ are nonempty. Note that $B'\cap X$ has at least two connected components as $0\notin B'$. We can decrease the radius of $B'$ further to the moment of $B'':=B(o,r'-\delta')$, when one of the intersections $B''\cap X^\pm$ becomes contained in a circle $C(o,r'-\delta')$. Then we choose any point in that intersection and we make $B''$ smaller again, keeping the tangency in the chosen point, until the disk is tangent to the second part of $X$ as well. In the end, we get the disk $B'''=B(o',r'')\subset B(o,r')$ with $o'\in M_X$. 
    \par Passing with the point $p$ on the curve $y=c'\,x$ to $0$ brings $r'\to 0$ and, consequently, $o'\to 0$. Finally $0\in \overline{M_X}$, which ends the proof.
    \begin{figure}[h]
    \includegraphics[width=8cm]{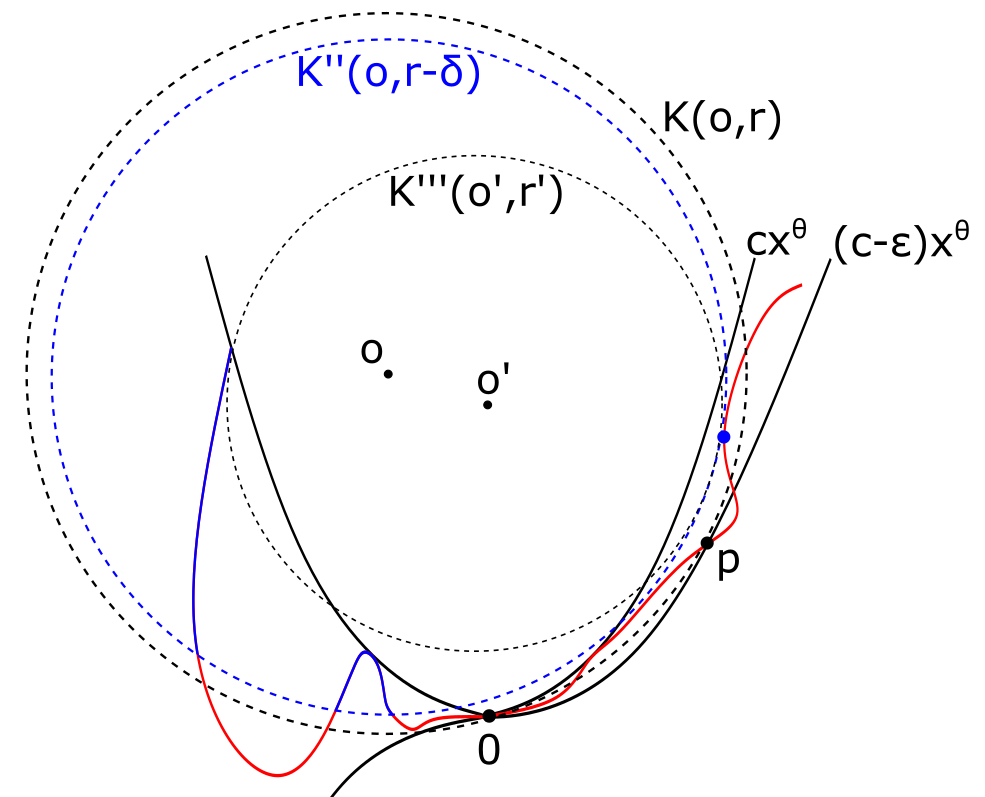}
    \centering
    \end{figure}
\end{proof}

\section{Higher dimensions and $\theta< 2$}

Similarly to \cite{BD} and \cite{Barx} we introduce the following `shape' function.
\begin{definition}
We define the \textit{shape function} of the germ $(X,0)$ in the direction $v\in V_0^C (X)$ by
$$
g_{X,v}(r) = \max \left\{ |t|\mid (z\oplus tv)\in X, \norm{z}=r\right\}
$$
where we decompose as earlier $\mathbb{R}^n=(\mathbb{R}v)^\bot\oplus \mathbb{R}v$.
\end{definition}
Of course, $g_{X,v}$ is a definable function and $(*)$ implies $g_{X,v}(0)=0$ (cf. Remark \ref{pierwsza}). Recall that $\mathrm{ord}_0 g_{X,v}$, whenever finite, is the exponent $\eta$ appearing in the expansion $g_{X,v}(r)=cr^\eta+o(r^\eta)$, where $c>0$ is a constant. It can be geometrically characterised by $\mathrm{ord}_0g_{X,v}=\sup\{\eta>0\mid g_{X,v}(r)\leqslant\mathrm{const.} r^\eta,\ \textrm{near}\ 0\}$ the least upper bound being attained.
\begin{lemma}\label{rzad}
In the context considered $\mathrm{ord}_0 g_{X,v} = \theta_v(X,0)$.
\end{lemma}
\begin{proof} Obviously, there is nothing to prove when $(X,0)=(\mathbb{R}^n,0)$ as both sides are $+\infty$. Let us consider the remaining case. The set $$E:=\{(r,z\oplus tv)\in [0,\varepsilon)\times X\mid \|z\|=r, g_{X,v}(\|z\|)=|t|\}$$ is definable and has an accumulation point at the origin. Therefore, there is a continuous definable selection $\gamma\colon [0,\varepsilon)\ni r\mapsto (r, z(r)\oplus \tau(r) v)\in E$. Then $z(r)$ is automatically parameterised by arc-length: $\norm{z(r)} = r$, and $\tau(r)$ being definable, we may assume it has a constant sign, say $\tau(r)\geqslant 0$. It follows that $g_{X,v}(r)=\tau(r)$ so that $\mathrm{ord}_0g_{X,v}=\mathrm{ord}_0\tau$. Clearly, $\tau(r)\leqslant \mathrm{const.}\|z(r)\|^{\theta_v(X,0)}$ so that $\mathrm{ord}_0\tau\geqslant \theta_v(X,0)$. On the other hand, given $z\oplus tv\in X$ near the origin, we have $$|t|\leqslant g_{X,v}(\|z\|)=\tau(\|z\|)\leqslant \mathrm{const.}\|z\|^{\mathrm{ord}_0 \tau}$$ whence $\theta_v(X,0)\geqslant \mathrm{ord}_0 \tau$ and we are done.
\end{proof}

\begin{lemma}\label{order}
   For $g_X(r):=\max\{\|y\|\mid (x,y)\in X, \|x\|=r\}$, $0\leq r\ll 1$, which is a definable function, there is $\mathrm{ord}_0 g_X=\theta(X,0)$. 
\end{lemma}
\begin{proof}
We may assume that $X$ is not contained entirely in $\mathbb{R}^k\times\{0\}^{n-k}$ so that $\theta(X,0)$ and $\mathrm{ord}_0 g_X$ are both finite.
The proof is similar to the previous one in that we consider the definable set
$$
E:=\{(r,x\oplus y)\in [0,\varepsilon)\times X\mid ||x||=r, g_X(||x||)=||y||\}
$$
that has an accumulation point at the origin. Then we use the continuous definable selection $\gamma(r)=(r,x(r)\oplus y(r))$ where $||x(r)||=r$ and $g_X(r)=||y(r)||$. Thus $\mathrm{ord}_0g_X=\mathrm{ord}_0 ||y(\cdot)||$. Since $x(r)\oplus y(r)\in X$, it follows that $||y(r)||\leq \mathrm{const.}||x(r)||^{\theta(X,0)}$ whence $\mathrm{ord}_0 g_X\geq \theta(X,0)$. On the other hand, for $(x,y)\in X$, $||y||\leq g_X(||x||)\leq \mathrm{const.}||x||^{\mathrm{ord}_0 g_X}$, whence $\theta(X,0)\geq \mathrm{ord}_0 g_X$.
\end{proof}

\begin{theorem}\label{kierunek1}
    Let $a\in X$ and $v\in V_a(X)$. If the directional reaching radius vanishes: $r_v(a)=0$, then $v$ is tangent to $M_X$ (i.e. $v\in C_aM_X$).
\end{theorem}
\begin{proof}
    Obviously, the reaching radius vanishes as well, $r(a)=0$, so that $a$ belongs to the closure of $M_X$ and the tangent cone $C_aM_X$ is well defined. Notice also that for any $t>0$ we have $a\notin m(a+tv)$. Therefore, we can pick points $a\neq y_t\in m(a+tv)$ and, since $a+tv\to a$ when $t\to 0^+$, regardless of their choice, we have $y_t\to a$, as $m(a)=\{a\}$ and the multifunciton of closest points is upper semi-continuous in the Painlev\'e-Kuratowski sense (see \cite{BD}).
    
     We will use and extend a part of the proof of Theorem 4.35 in \cite{BD}. Denote \[v(t):=\frac{a+tv-y_t}{\|a+tv-y_t\|},\] we obviously get $r_{v(t)}(y_t)$ positive. We claim now that $v(t)\to v$ as $t\to 0^+$.

    Firstly, let us recall that a normal vector $u\in V_a(X)$ is called \textit{proximal at $a$,} if for some $\delta>0$, the point $a$ is the closest point to $a+\delta u$ in $X$. This is equivalent to saying that for any $x\in X$, we have $\langle x-a,u\rangle\leqslant(1/2\delta)\|x-a\|^2.$ 

    In our case, we know that the vectors $v(t)$ are proximal at $y_t$.

    Let us now assume, without loss of generality, that the unit vectors $w(t):=(y_t-a)/\|y_t-a\|$ converge to some $w\in C_aX$. Let us denote $\beta_t:=\langle w(t),v\rangle$.

    Then we observe that $\beta_t\to 0.$ Indeed, $w\in C_aX$ implies $\langle v, w\rangle\leqslant 0$. If there were $\langle v,w\rangle <0$, we would have $\langle v,w(t)\rangle <0$, for $t$ small enough. The inequality 
    \[ \|a+tv-y_t\|^2=t^2+\|y_t-a\|^2-2t\|y_t-a\|\beta_t>t^2\] 
    would lead to a contradiction, since the points $y_t$ were assumed to be closer to $a+tv$ than $a$.

    Now, $v(t)$ being proximal at $y_t$ with $\delta = d(a+tv,X)$, we can write 
    \[\langle a-y_t,v(t)\rangle\leqslant\frac{1}{2d(a+tv,X)}\|a-y_t\|^2\]
    which is equivalent to 
    \begin{align*}\langle a-y_t,a+tv-y_t\rangle  &\leqslant \frac{1}{2}\|a-y_t\|^2   \\
    \|a-y_t\|^2+ t\langle a-y_t,v\rangle        &\leqslant \frac{1}{2}\|a-y_t\|^2   \\
    \|a-y_t\|^2               &\leqslant 2 t\langle y_t-a,v\rangle =                \\
                                                &=2t\|y_t-a\|\beta_t
    \end{align*}
    It follows that $\|a-y_t\|/t\to 0$.

    It remains to show that $\langle v(t),v\rangle$ converges to $1$. Obviously $\langle v(t),v\rangle\leq 1$, both vectors being unitary. Since $\|a+tv-y_t\|<t$, we have
    \begin{align*}
        \langle v(t),v\rangle &= \frac{1}{\|a+tv-y_t\|}\langle a+tv-y_t,v\rangle = \\
                            &=\frac{1}{\|a+tv-y_t\|}(\langle a-y_t,v\rangle +t)> \\
                            &>\frac{1}{t}(\langle a-y_t,v\rangle +t)=\frac{\langle a-y_t,v\rangle}{t} +1. 
    \end{align*}
    But since $\|a-y_t\|/t\to 0$ we have proved the claim.

    Next, we argue that the directional reaching radius $r_{v(t)}(y_t)\to 0$ as $t\to 0^+$. If we assume otherwise, then, without loss of generality, there is a constant $\varrho>0$ such that $r_{v(t)}(y_t)>\varrho$ for all $t$. Consequently, the open balls $B(y_t+\varrho\, v(t),\varrho)$ do not meet $X$ and converge to the closure of $B(a+\varrho\,v,\varrho).$ By the convergence, the latter open ball does not meet $X$ either. But this contradicts $r_v(a)=0$.

    In particular $r_{v(t)}(y_t)$ is finite for $t$ small enough. Hence, we can pick a sequence of points $\xi(t):=r_{v(t)}(y_t)v(t)+y_t \in \overline{M_X}$. For such a sequence, we have   
    \[\frac{\xi(t)-a}{\|\xi(t)-a\|} \to v.\]
    Indeed, we can rewrite $\xi(t)-a = (r_{v(t)}(y_t)-\|tv-y_t+a\|)v(t)+tv.$ Now since $(\xi(t)-a)/\|\xi(t)-a\|$ is normalised, $v(t)\to v$ and $r_{v(t)}(y_t)-\|tv-y_t+a\|\geqslant 0$ we can see that convergence holds.
    
\end{proof}

We are now ready for the general criterion.

\begin{theorem}\label{kierunek}
    If $\theta_v(X,0)<2$, then $0\in\overline{M_X}$. Actually, $$0\in \overline{M_X\cap\{(z\oplus tv)\in (\mathbb{R} v)^\bot\oplus \mathbb{R}v\mid |t|>\mathrm{const.} \|z\|^{\theta_v(X,0)}\}}.$$ Moreover, $\dim C_0M_X\cap\mathbb{R}v=1$, i.e. either $v$ or $-v$ is tangent to $M_X$ at the origin. 
\end{theorem}
\begin{proof}
Consider the selection $\gamma(r)=(r, z(r)\oplus \tau(r) v)$ from the proof of Lemma \ref{rzad}, with $\tau(r)\geqslant 0$. Thanks to the definability of $\xi(r):=z(r)\oplus \tau(r) v\in X$, which is non-constant, we know that the tangent cone to its image $C_0(\xi) = \mathbb{R}_+w$, for some $w\in C_0X\subset (\mathbb{R}v)^{\perp}$. Thus, $\langle v,w\rangle = 0$. Incidentally, $\mathbb{R}_+w$ is the half-line tangent to the image of $z(r)$.

If we prove that the directional reaching radius $r_v(0)=0$, then we get 
$0\in \overline{M_X}$ and $\dim C_0M_X\cap\mathbb{R}v=1$ by Theorem \ref{kierunek1}.

    Let us consider $\widehat{\tau}\colon \mathbb{R}_+w\ni sw \to \tau(s)\in \mathbb{R}_t$, then $\mathrm{ord}_0\widehat{\tau} = \mathrm{ord}_0 \tau$ and the latter coincides with $\theta_v(X,0)$ as was shown in the proof of Lemma \ref{rzad}. Therefore, $\widehat{\tau}$ is superquadratic in the sense of \cite{BD} Section 3.3, which means that its graph $\Gamma_{\widehat{\tau}}$ enters any disc $D((0,\rho),\rho)$, $0<\rho\ll 1$, cf. \cite{BD} Lemma 3.17.  This graph is the projection onto the plane $\mathbb{R}w\oplus \mathbb{R}v$ of the graph $\Gamma_\xi$ of $\xi$. Eventually, by construction, we conclude that the graph of $\xi$ enters any ball $B(\rho v,\rho)$, $0<\rho\ll 1$, so that $0\notin m(\rho v)$, whence $r_v(0)=0$, as required.
\end{proof}

\begin{corollary}\label{wniosek}
    If $\theta(X,0)<2$, then $0\in\overline{M_X}$.
\end{corollary}
\begin{proof}
    It follows directly from Theorems \ref{theta} and \ref{kierunek}, since there must be a direction $v$ in which $\theta_v(X,0)<2$.
\end{proof}
\begin{example}\label{horn}
    Consider the horn $X\colon x^3=y^2+z^2$. Its medial axis consists of the positive part of the $x$-axis together with a punctured revolution surface in the half-space $\{x\leqslant 0\}$. Since $C_0X=[0,+\infty)\times\{0\}^2$, the camber directions are $\{0\}\times\mathbb{R}^2\cap\mathbb{S}^2$ and it is easy to see that $\theta(X,0)=\frac{3}{2}$. The Theorem applies and indeed $0\in\overline{M_X}$. Note however, that we completely miss the $x$-axis part of the medial axis which also reaches the origin. The Theorem detects far less obvious part of the medial axis. Observe also, that in the case of a horn we cannot apply the Tangent Cone Criterion, nor any previously developed idea of superquadracity as the origin is a singular point.
\end{example}

\section{Higher dimensions and $\theta\geqslant 2$}

\begin{proposition}\label{higher geq 2} There is a counterexample to the statement that for $n\geqslant 3$ and a definable, non-superquadratic $(X,0)\subset\mathbb{R}^n$, there is $0\notin\overline{M_X}$.\end{proposition}

\begin{proof} 
Let us consider $\mathbb{R}^3$. Let $X$ be defined as follows:
\begin{align*}
X':&= \left(\bigcup_{t\in \mathbb{R}} B((t,0,0), |t|)\cap \{(x,y,z)\mid  x=t\}\right) \cup\\
&\cup \left(\bigcup_{t\in \mathbb{R}} B((0,t,0), |t|)\cap \{(x,y,z)\mid  y=t\}\right).
\end{align*}
Next, we `cut off' this set:
$$
X := X'\cap \{(x,y,z)\mid  |z|\leqslant x^2+y^2\}
$$
obtaining a non-superquadratic, definable $X$ consisting of four identical parts. Let us notice that
$$
\{(x,y,z)\mid  (x,y)\neq 0\ \land\ z\neq 0\ \land \ |x|=|y|\} \cup \left\{(0,0,t)\mid\ |t|>\frac{1}{4}\right\}\subset M_X
$$
$$
\left\{(0,0,t)\mid |t|\leqslant\frac{1}{4}\right\}\cap M_X = \varnothing
$$
so the figure described above keeps $0\in\overline{M_X}$, moreover $M_X$ is `glued' to the vertical $z$-direction at zero, which is an axis that it does not contain.
\end{proof}

\begin{figure}[h]
\includegraphics[width=8cm]{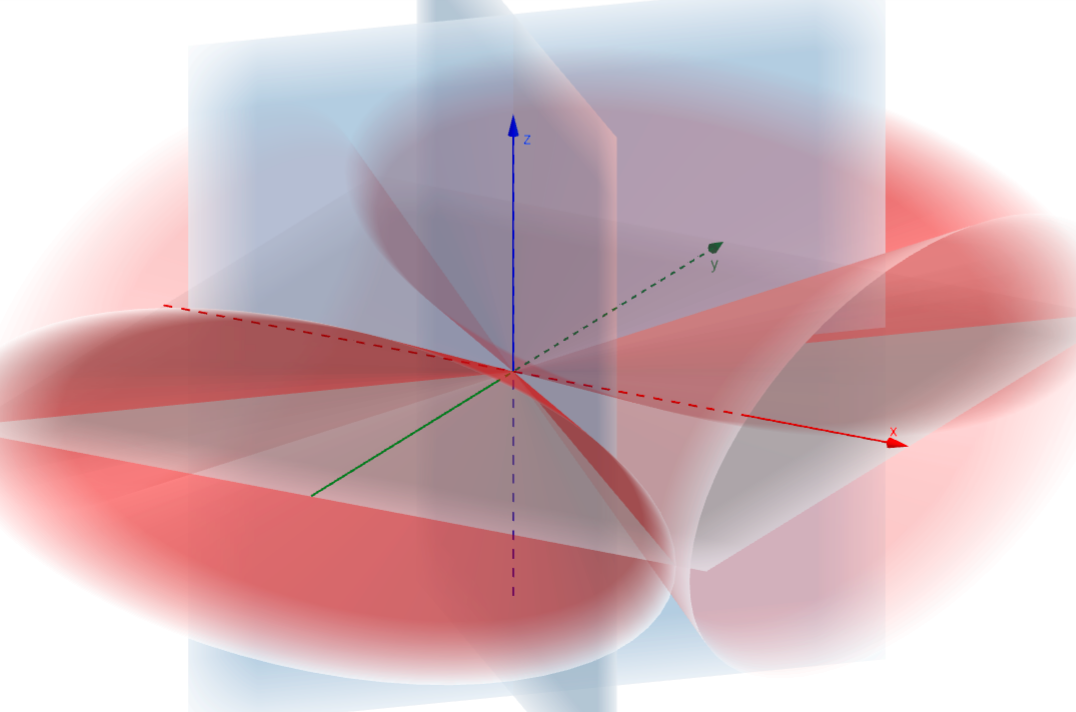}
\centering
\end{figure}

\section{Acknowledgements}

During the preparation of this article Dominik Bysiewicz was a participant of the tutoring programme under the Excellence Initiative at the Jagiellonian University.

\end{document}